\documentclass{amsart}%kh
\usepackage{amssymb}
\usepackage{amsmath}
\usepackage{latexsym}
\usepackage{eepic}
\usepackage{epsfig}
\usepackage{graphicx}
\usepackage{pb-diagram,lamsarrow,pb-lams,amscd}
\usepackage{eufrak}
\usepackage{mathrsfs}
\usepackage[dutch,english]{babel}
\usepackage[all,cmtip]{xy}

\DeclareMathAlphabet{\mathpzc}{OT1}{pzc}{m}{it}
\newtheorem{theorem}{Theorem}[section]
\newtheorem{theorem-definition}[theorem]{Theorem-Definition}
\newtheorem{lemma-definition}[theorem]{Lemma-Definition}
\newtheorem{definition-prop}[theorem]{Proposition-Definition}

\newtheorem{prop}[theorem]{Proposition}
\newtheorem{lemma}[theorem]{Lemma}
\newtheorem{cor}[theorem]{Corollary}

\newtheorem{conjecture}[theorem]{Conjecture}

\newenvironment{remark}{\vspace{4pt}\noindent\textbf{Remark.}}{\qed\vspace{4pt}}

\newcommand{\LL}{\ensuremath{\mathbb{L}}}

\newcommand{\Z}{\ensuremath{\mathbb{Z}}}
\newcommand{\Q}{\ensuremath{\mathbb{Q}}}

\newcommand{\A}{\ensuremath{\mathbb{A}}}

\newcommand{\G}{\ensuremath{\mathbb{G}}}

\renewcommand{\A}{\ensuremath{\mathbb{A}}}

\numberwithin{equation}{section} \hyphenpenalty=6000
\begin{document}
\title{Trace formula for component groups of N\'eron models}
\author[Johannes Nicaise]{Johannes Nicaise}
\address{Universit\'e Lille 1\\
Laboratoire Painlev\'e, CNRS - UMR 8524\\ Cit\'e Scientifique\\59655 Villeneuve d'Ascq C\'edex \\
France} \email{johannes.nicaise@math.univ-lille1.fr}
\begin{abstract}
We study a trace formula for tamely ramified abelian varieties $A$
over a complete discretely valued field, which expresses the Euler
characteristic of the special fiber of the N\'eron model
 of $A$ in terms of the Galois action on the $\ell$-adic cohomology of $A$. If $A$ has purely additive reduction, the trace
formula yields a cohomological interpretation for the number of
connected components of the special fiber of the N\'eron model.
\end{abstract}
 \maketitle
\section{Introduction}
We denote by $R$ a complete discrete valuation ring, by $K$ its
quotient field, and by $k$ its residue field. We assume that $k$
is algebraically closed, and we denote its characteristic exponent
by $p$. We denote by $K^s$ a separable closure of $K$, and by
$K^t$ the tame closure of $K$ in $K^s$. We fix a topological
generator $\varphi$ of the tame monodromy group $G(K^t/K)$ and a
prime $\ell$ invertible in $k$. We denote by
$\chi_{\acute{e}t}(X)$ the $\ell$-adic Euler characteristic of a
$k$-variety $X$ (it is independent of $\ell$). For each $R$-scheme
$Z$, we denote by $Z_s$ its special fiber over $k$.

For each $K$-variety $X$, we denote by $H(X\times_K K^s,\Q_\ell)$
the graded $\ell$-adic cohomology space $\oplus_{i\geq
0}H^i(X\times_K K^s,\Q_\ell)$. We say that $X$ is cohomologically
tame if the wild inertia $P\subset G(K^s/K)$ acts trivially on
$H(X\times_K K^s,\Q_\ell)$. In this case, we have
$$H(X\times_K
K^t,\Q_\ell)\cong H(X\times_K K^s,\Q_\ell)^P=H(X\times_K
K^s,\Q_\ell)$$

If $X$ is a smooth and proper $K$-variety, and $\mathcal{X}$ is a
weak N\'eron model for $X$ over $R$ \cite[3.5.1]{neron}, then  the
value $\chi_{\acute{e}t}(\mathcal{X}_s)$ is independent of the
choice of weak N\'eron model, since it is a specialization of the
motivic Serre invariant $S(X)$ of $X$ (see Section
\ref{subsec-chevalley}). In \cite[6.3]{Ni-tracevar} we established
a trace formula for smooth and proper $K$-varieties $X$ that
satisfy a certain geometric tameness condition. The trace formula
states, in particular, that
$$\chi_{\acute{e}t}(\mathcal{X}_s)=Trace(\varphi\,|\,H(X\times_K
K^t,\Q_\ell)):=\sum_{i\geq 0}(-1)^i
Trace(\varphi\,|\,H^i(X\times_K K^t,\Q_\ell))$$ The value
$\chi_{\acute{e}t}(\mathcal{X}_s)$ can be seen as a measure for
the set of rational points on $X$, and the trace formula gives a
cohomological interpretation for this measure.

If $p=1$ the geometric tameness condition is void, but
unfortunately, it is quite strong if $p>1$. In
\cite[\S\,7]{Ni-tracevar} we showed that if $X$ is a geometrically
connected smooth projective $K$-curve, we can weaken the condition
to cohomological tameness, unless $X$ is a genus one curve with
additive reduction and without $K^t$-point. It seems plausible
that
 the trace formula holds
for any cohomologically tame, smooth, proper and geometrically
connected $K$-variety $X$ that has a rational point over $K^t$.

 The purpose of the present note is to
prove this for abelian varieties. Some of the results are
certainly known to experts (even though we could not find them in
the literature), but we believe that the trace formula provides a
valuable context which casts new light on the results presented
here, inserting them in a general theory. Conversely, the case of
abelian varieties constitutes a first step towards a ``motivic''
proof of the general case, avoiding the use of resolution of
singularities.

Let $A$ be an abelian variety over $K$. We denote by $\mathcal{A}$
its N\'eron model over $R$, and by $\phi_A$ the cardinality of the
group of connected components
%$\Phi_A=
 $\Phi_A=\mathcal{A}_s/\mathcal{A}_s^o$. We denote by $\phi_A'$ the
prime-to-$p$ part of $\phi_A$, and by $(\phi_A)_{q}$ the
$q$-primary part, for each prime number $q$. We say that $A$ is
tamely ramified if it is cohomologically tame; this is equivalent
to the property that $A$ acquires semi-abelian reduction on a
finite tame extension of $K$. The toric, unipotent and abelian
rank of $A$ are by definition the toric, unipotent and abelian
rank of $\mathcal{A}_s^o$.
%We denote them by $t_A,\,u_A$ and
%$a_A$, respectively.

We consider the following property for the abelian $K$-variety
$A$.
$$\begin{array}{ll} \mbox{(Trace)}&
\chi_{\acute{e}t}(\mathcal{A}_s)=Trace(\varphi\,|\,H(A\times_K
K^t,\Q_\ell))
%\\ \mbox{(Tame)}& \mbox{If $A$ has purely additive reduction, then
%$\phi_A= \phi'_A$.}
%\\ \mbox{(Tame')}& \mbox{If the toric rank of $A$ vanishes, then
%$\phi_A= \phi'_A$.}
\end{array}$$
%
%It is obvious that (Tame') implies (Tame) for any abelian
%$K$-variety $A$.

In \cite[\S\,7.3]{Ni-tracevar} we formulated the following
conjecture.

\begin{conjecture}\label{trace-conj}
The trace formula (Trace)  holds for any tamely ramified abelian
$K$-variety $A$.
\end{conjecture}

In the present article, we will prove this conjecture.
%
%show that properties (Trace) and (Tame) are equivalent for any
%tamely ramified abelian variety $A$. In particular, (Trace) holds
%for any tamely ramified abelian variety which does not have purely
%additive reduction.
 In fact, we show that both sides of the expression (Trace) vanish
unless $A$ has purely additive reduction (Proposition
\ref{prop-notadd}), and that in the latter case, both sides equal
$\phi_A$
%the left hand side equals $\phi_A$ and the
%right hand side equals $\phi'_A$
 (Theorem \ref{theo-trace}).

%Moreover, we strengthen Conjecture \ref{trace-conj} as follows:
%\begin{conjecture}\label{tame-conj}
%Property (Tame') holds for any tamely ramified abelian $K$-variety
%$A$.
%\end{conjecture}
%Conjecture \ref{tame-conj} is known in each of the following
%cases:
%\begin{enumerate}
%\item  $k$ has characteristic zero (this is trivial) \item $A$ has
%potentially good reduction (Proposition \ref{prop-good}) \item $A$
%is a Jacobian (Proposition \ref{tame-jac})
%\end{enumerate}
%(see also the Remark following Proposition \ref{prop-good} on the
%general case). As a consequence, we see that Conjecture
%\ref{trace-conj} holds in each of these cases.
%
%If $k$ has characteristic zero, then Conjecture \ref{trace-conj}
%was proven in \cite[6.5]{Ni-tracevar} in a much more general
%setting, using computations on a resolution of singularities; the
%methods in the present paper do not rely on resolution of
%singularities.
%
%Note that properties (Trace), (Tame) and (Tame') are stable under
%products of tamely ramified abelian varieties; in particular, they
%hold for finite products of tamely ramified abelian varieties
%which are of the forms (2) or (3) above.
%
%\begin{remark}
%The condition that the toric rank of $A$ vanishes can certainly
%not be omitted in Conjecture \ref{tame-conj}. Consider, for
%instance, the case where $p>1$ and $A$ is an elliptic curve of
%reduction type $I_p$. Then $A$ is tamely ramified (it has
%semi-abelian reduction) but $\phi_A=p$. Such a curve always
%exists, and it can be realized as a Tate curve.
%\end{remark}

\section{The trace formula}
\subsection{Chevalley decomposition in the Grothendieck
ring}\label{subsec-chevalley} We denote by $K_0(Var_k)$ the
Grothendieck ring of $k$-varieties. See, for instance,
\cite[2.1]{Ni-tracevar} for its definition. As usual, we denote by
$\LL$ the class in $K_0(Var_k)$ of the affine line $\A^1_k$.
\begin{prop}\label{decomp}
Let $G$ be a smooth connected commutative algebraic $k$-group and
consider its Chevalley decomposition
$$\begin{CD}0@>>> (L=U\times_k T)@>>> G@>\pi>> B@>>> 0\end{CD}$$
with $U$ unipotent, $T$ a torus, and $B$ an abelian variety. If we
denote by $u$ and $t$ the dimensions of $U$, resp. $T$, then
$$[G]=\LL^{u}(\LL-1)^t[B]$$
in $K_0(Var_{k})$.
\end{prop}
\begin{proof}
As a $k$-variety, $U$ is isomorphic to $\A^u_k$
\cite[VII\,n$^\mathrm{o}$\,6]{serre-groupalg}. By the scissor
relations in the Grothendieck ring, it suffices to show that $\pi$
is a $L$-torsor w.r.t. the Zariski-topology. But $\pi$ is a
$L$-torsor w.r.t. the $fppf$ topology, and hence also w.r.t. the
Zariski topology because $L$ is a successive extension of $\G_m$
and $\G_a$ \cite[III.3.7+4.9]{Milne}.
%Consider the Chevalley decomposition
%$$\begin{CD}0@>>> U\times_k T@>>> G@>\pi>> B@>>> 0\end{CD}$$
%with $U$ unipotent, $T$ a torus, and $B$ an abelian variety. Since
%$B$ is a homogeneous space under $G$, the complex
%$R\pi_{!}(\Q_\ell)$ is lisse on $B$, and
%$$\chi_{\acute{e}t}(G)=\chi_{\acute{e}t}(A,R\pi_{!}(\Q_\ell))=\chi_{\acute{e}t}(A)\cdot \chi_{\acute{e}t}(U)\cdot \chi_{\acute{e}t}(T)$$
%by Leray's spectral sequence, proper base change, and
%\cite{Illusie}. Hence, $\chi_{\acute{e}t}(G)$ vanishes unless $G=U$, and
%since $U$ is isomorphic to an affine space, we find
%$\chi_{\acute{e}t}(U)=1$.
\end{proof}
\begin{cor}
If $G$ is a smooth connected commutative algebraic $k$-group, then
$\chi_{\acute{e}t}(G)=1$ if $G$ is unipotent, and
$\chi_{\acute{e}t}(G)=0$ else.
\end{cor}

\begin{cor}\label{varphi}
For each abelian variety $A$ over $K$, we have
$\chi_{\acute{e}t}(\mathcal{A}_s)=\phi_A$ if $A$ has purely
additive reduction, and $\chi_{\acute{e}t}(\mathcal{A}_s)=0$ else.
\end{cor}

Recall that the motivic Serre invariant $S(X)$ of a smooth and
proper $K$-variety $X$ is defined as
$$S(X)=[\mathcal{X}_s]\in K_0(Var_k)/(\LL-1)$$ where $\mathcal{X}$ is
any weak N\'eron model of $X$ over $R$. This definition is
independent of the choice of weak N\'eron model
\cite[4.5.3]{motrigid}\cite[5.1]{Ni-tracevar}. The motivic Serre
invariant may be seen as a measure for the set of rational points
on $X$, and the trace formula in \cite[6.3]{Ni-tracevar} gives a
cohomological interpretation of this measure.

\begin{cor}\label{cor-motserre}
Let $A$ be an abelian variety over $K$, denote by $u$ and $t$ its
unipotent, resp. toric rank, and denote by $B$ the abelian
quotient in the Chevalley decomposition of $\mathcal{A}_s^o$. Then
we have the following equalities in $K_0(Var_k)/(\LL-1)$~:
$$S(A)=\left\{\begin{array}{ll}0&\mbox{ iff }t>0\\ \phi_A\cdot [B] &\mbox{ iff }t=0 \end{array}\right.$$
\end{cor}
\begin{proof}
We only have to show that $\phi_A\cdot [B]\neq 0$ in
$K_0(Var_k)/(\LL-1)$. Consider the Poincar\'e polynomial
$$P(\cdot;T):K_0(Var_k)\rightarrow \Z[T]:[X]\mapsto P(X;T)$$
from \cite[2.10]{Ni-tracevar}. It is a morphism of rings, mapping
the class $[Y]$ of a smooth and proper $k$-variety $Y$ to the
polynomial
$$P(Y;T)=\sum_{i\geq 0}(-1)^i b_i(Y)T^i$$ with $b_i(Y)$ the $i$-th
$\ell$-adic Betti number of $Y$ (it is independent of $\ell$). We
have $P(\LL;T)=T^2-1$, so evaluating at $T=-1$ we obtain a ring
morphism
$$P(\cdot;-1):K_0(Var_k)/(\LL-1)\rightarrow \Z:[X]\mapsto P(X;-1)$$
We have
$$P(\phi_A\cdot [B];-1)=\phi_A\cdot \sum_{i\geq 0}b_i(B)=\phi_A\cdot
4^{\mathrm{dim}(B)}\neq 0$$
\end{proof}

%\subsection{Motivic Serre invariant}
%If $X$ is a proper and smooth $K$-variety, its motivic Serre
%invariant $S(X)$ is defined as the class $[\mathcal{X}_s]$ of the
%special fiber of a weak N\'eron model $\mathcal{X}/R$ of $X$ in
%$K_0(Var_k)/(\LL-1)$. Using motivic integration, one shows that
%this definition does not depend on $\mathcal{X}$ (see
%\cite{motrigid} and \cite{Ni-tracevar}). In particular,
%$\chi_{\acute{e}t}(S(X))=\chi_{\acute{e}t}(\mathcal{X}_s)$ only depends on $X$.
%It can be considered as a measure for the set of $K$-points on $X$
%(the motivic Serre invariant $S(X)$ vanishes if $X(K)=\emptyset$).
%For an abelian $K$-variety $A$, we have $S(A)=[\mathcal{A}_s]$.
%\begin{lemma}
%Let $A$ be an abelian variety over $K$. If we denote by $u$ and
%$t$ the unipotent, resp. toric rank of $A$, and by $B$ the abelian
%quotient in the Chevalley decomposition of $\mathcal{A}^o_s$, then
%$$S(A)=\left\{ \begin{array}{ll}0&\mbox{ iff }
%t_A>0\\ \phi_A\cdot \LL^{u}[B] &\mbox{ iff }t=0
%\end{array}\right.
%$$ in $K_0(Var_k)/(\LL-1)$. In particular,
%$\chi_{\acute{e}t}(S(A))=\phi_A$ iff $A$ has purely additive reduction,
%and $\chi_{\acute{e}t}(S(A))=0$ else.
%\end{lemma}
%\begin{proof}
%This follows immediately from Proposition \ref{decomp}.
%\end{proof}

\subsection{Trace formula for component groups}
Assume that $A$ is an abelian $K$-variety of dimension $g$. We
denote by $P_{\varphi}(T)$ the characteristic polynomial
$$P_{\varphi}(T)=det(T\cdot Id-\varphi\,|\,H^1(A\times_K K^t,\Q_\ell))$$
of $\varphi$ on $H^1(A\times_K K^t,\Q_\ell)$. The polynomial
$P_{\varphi}(T)$ belongs to $\Z[T]$, and it is independent of
$\ell$, by \cite[2.10]{lorenzini-charpol}. By quasi-unipotency of
the $G(K^s/K)$-action on $T_\ell A$ \cite[IX.4.3]{sga7a} we know
that the zeroes of $P_{\varphi}(T)$ are roots of unity, so that
$P_{\varphi}(T)$ is a product of cyclotomic polynomials. Since the
pro-$p$-part of $G(K^t/K)$ is trivial, the orders of the zeroes of
$P_{\varphi}(T)$ as roots of unity are prime to $p$.

If $A$ is tamely ramified, it follows immediately from the
canonical isomorphism
$$H(A\times_K K^t,\Q_\ell)\cong\bigwedge
H^1(A\times_K K^t,\Q_\ell)$$ that we have
$$Trace(\varphi\,|\,H(A\times_K
K^t,\Q_\ell))=P_{\varphi}(1)$$

\begin{prop}\label{prop-notadd}
If $A$ is a tamely ramified abelian $K$-variety that does not have
purely additive reduction, then
$$\chi_{\acute{e}t}(\mathcal{A}_s)= Trace(\varphi\,|\,H(A\times_K
K^t,\Q_\ell))=0$$ and Conjecture \ref{trace-conj} holds for $A$.
\end{prop}
\begin{proof}
It is well-known that $P_{\varphi}(1)=0$ iff $A$ does not have
purely additive reduction (see for instance
\cite[1.3]{Lenstra-Oort}) so the right hand side of trace formula
vanishes. The left hand side vanishes as well, by Corollary
\ref{varphi}.
\end{proof}

%\begin{lemma}
%If $q$ is a prime, $M$ a free $\Z_q$-module of finite type, and
%$\sigma$ an automorphism of the $\Q_q$-vector space
%$V=M\otimes_{\Z_q}\Q_q$ such that $\sigma(M)\subset M$, then
%$N=M/\sigma(M)$ is torsion and its cardinality equals
%$$|det(\sigma\,|\,V)|_{q}^{-1}$$ with $|\cdot|_{q}$ the $q$-adic absolute value.
%\end{lemma}
%\begin{proof}
%It is clear that $N$ is torsion, since $\sigma$ is an automorphism
%on $V$.
%\end{proof}

In order to investigate the case where $A$ has purely additive
reduction, we'll need some elementary lemmas.
\begin{lemma}\label{cyclo}
Fix an integer $d>1$ and let $\Phi_d(T)\in \Z[T]$ be the
cyclotomic polynomial whose roots are the primitive $d$-th roots
of unity. Then $\Phi_d(1)\in \Z_{>0}$ and $\Phi_d(1)| d$.
\end{lemma}
\begin{proof}
We proceed by induction on $d$. For $d=2$ the result is clear, so
assume that it holds for each value $d'$ with $1<d'<d$. We can
write
$$\frac{T^d-1}{T-1}=\prod_{e|d,\,e>1}\Phi_e(T)$$ and evaluating at
$T=1$ we get
$$d=\prod_{e|d,\,e>1}\Phi_e(1)$$ so $\Phi_d(1)|d$.
By the induction hypothesis, $\Phi_e(1)>0$ for $1<e<d$, so
$\Phi_d(1)>0$ as well.
\end{proof}
\begin{lemma}\label{linalg}
Let $q$ be a prime, $M$ a free $\Z_q$-module of finite type, and
$\alpha$ an endomorphism of $M$. Then $M/\alpha M$ is torsion iff
$\alpha$ is an automorphism on $M\otimes_{\Z_q}\Q_q$. In this
case, the cardinality $\sharp (M/\alpha M)$ of $M/\alpha M$
satisfies
$$\sharp (M/\alpha M)=|det(\alpha\,|\,M\otimes_{\Z_q}\Q_q)|_q^{-1}$$
where $|\cdot|_q$ denotes the $q$-adic absolute value.
\end{lemma}
\begin{proof}
The module $M/\alpha M$ is torsion iff  $(M/\alpha
M)\otimes_{\Z_q}\Q_q=0$, i.e. iff $\alpha$ is a surjective and,
hence, bijective endomorphism on $M\otimes_{\Z_q}\Q_q$. In this
case, we have
$$M/\alpha M\cong \Z_q/q^{c_1}\Z_q\oplus\ldots \oplus \Z_q/q^{c_r}\Z_q$$
where $q^{c_1},\ldots,q^{c_r}$ are the invariant factors of
$\alpha$ on $M$. Since $det(\alpha\,|\,M\otimes_{\Z_q}\Q_q)$
equals $q^{c_1+\ldots+c_r}$ times a unit in $\Z_q$, we find
$$\sharp (M/\alpha M)=q^{c_1+\ldots+c_r}=|det(\alpha\,|\,M\otimes_{\Z_q}\Q_q)|_q^{-1}$$
\end{proof}
%\begin{lemma}\label{toric0}
%Let  $q$ be a prime different from $p$. Denote by $T_qA$ the
%(tame) $q$-adic Tate module of $A$, and put
%$$M_qA=T_qA/ker(Id-\varphi\,|\,T_qA)$$ The following properties
%are equivalent: \begin{enumerate} \item the toric rank of $A$
%vanishes \item the value $1$ is not an eigenvalue of $\varphi$ on
%$M_qA\otimes_{\Z_q}\Q_q$ \item the $\Z_q$-module
%$M_qA/(Id-\varphi)M_qA$ is torsion, and the natural morphism of
%$\Z_q$-modules
%$$T_qA/(Id-\varphi)T_qA\rightarrow M_qA/(Id-\varphi)M_qA$$ is an
%isomorphism on torsion parts. \item $Q_{\varphi}(T)$ is the
%characteristic polynomial of $\varphi$ on
%$M_qA\otimes_{\Z_q}\Q_q$.
%\end{enumerate}
%\end{lemma}
\begin{theorem}\label{theo-trace}
If $A$ is an abelian $K$-variety with
%toric rank
%zero, then
%$$Q_{\varphi}(1)=\phi'_A $$
%In particular, if $A$ has
 purely additive reduction, then
 $$P_{\varphi}(1)=\phi'_A$$
 If, moreover, $A$ is tamely ramified, then
$$Trace(\varphi\,|\,H(A\times_K
K^t,\Q_\ell))=\phi'_A=\phi_A$$
\end{theorem}
\begin{proof}
%By \cite[1.3]{Lenstra-Oort}, we have
%$P_{\varphi}(T)=Q_{\varphi}(T)$ iff $A$ has purely additive
%reduction, so the second assertion follows from the first.
 %Arguing as in \cite[\S\,2]{edix-primep} we see that,
By \cite[1.8]{liu-lorenzini}, the component group $\Phi_A$ is
killed by $[K':K]^2$, with $K'$ the minimal extension of $K$ where
$A$ acquires semi-abelian reduction. If $A$ is tamely ramified,
then $[K':K]$ is prime to $p$, so that $\phi_A=\phi'_A$.

Hence, it suffices to prove the first assertion. Let $q$ be a
prime different from $p$, and denote by $T^t_qA=(T_q A)^P$ the
tame $q$-adic Tate module of $A$. By \cite[IX.11.3.8]{sga7a}, the
$q$-primary part $(\phi_A)_{q}$ of $\phi_A$ equals the cardinality
of the torsion part of
$$H^1(G(K^s/K),T_qA)\cong H^1(G(K^t/K),T^t_{q}A)\cong T^t_q A/(Id-\varphi)T^t_q A$$
 Since $A$ has purely additive reduction, $1$ is not an eigenvalue
 of $\varphi$ on $V^t_qA=T^t_qA\otimes_{\Z_q}\Q_q$ by \cite[1.3]{Lenstra-Oort},
 so by Lemma \ref{linalg} the $\Z_q$-module
$T^t_q A/(Id-\varphi)T^t_q A$ is torsion, and
 its cardinality is given by
$$|det(1-\varphi\,|\,V^t_qA)|_q^{-1}$$ where $|\cdot|_q$ denotes the
$q$-adic absolute value. Since $P_{\varphi}(T)$ is independent of
$\ell$, we find
\begin{equation}\label{eq}(\phi_A)_{q}=|P_{\varphi}(1)|_{q}^{-1}\end{equation}
for each prime $q\neq p$. Moreover, the characteristic of $k$ does
not divide the order (as a root of unity) of any zero of
$P_{\varphi}(T)$, so if $p>1$, we have $|P_{\varphi}(1)|_p=1$ by
Lemma \ref{cyclo}.

Hence, taking the product of (\ref{eq}) over all primes $q\neq p$,
we get
$$\phi'_A=\prod_{q\neq p}|P_{\varphi}(1)|_{q}^{-1}=\prod_{r\,\mathrm{prime}}|P_{\varphi}(1)|_{r}^{-1}=|P_{\varphi}(1)|=P_{\varphi}(1)$$
where the last equality follows from Lemma \ref{cyclo}. This
concludes the proof.
\end{proof}
As a consequence, we obtain a proof of Conjecture
\ref{trace-conj}.
\begin{cor}
The trace formula (Trace) holds for every tamely ramified abelian
$K$-variety $A$.
\end{cor}
\begin{proof}
Combine Corollary \ref{varphi}, Proposition \ref{prop-notadd}, and
Theorem \ref{theo-trace}.
\end{proof}
\begin{cor}
If $A$ has purely additive reduction, then $\phi_A'$ is invariant
under isogeny.
\end{cor}
\begin{remark}
In the proof of Theorem \ref{theo-trace}, we invoked
\cite[1.8]{liu-lorenzini}. The proof of \cite[1.8]{liu-lorenzini}
is based on \cite[5.6+9]{bosch-xarles}. Unfortunately, it is known
that \cite[4.2]{bosch-xarles} is not correct (see
\cite[4.8(b)]{chai}), and the proofs of \cite[5.6+9]{bosch-xarles}
rely on this result. However, the authors of \cite{liu-lorenzini}
have informed me that they've written a proof of
\cite[1.8]{liu-lorenzini} that avoids the use of the erroneous
part of \cite[4.2]{bosch-xarles}. The proof can be found in
\cite{liu-lorenzini-complement}.
\end{remark}

\section{The monodromy zeta function}
 To conclude
this note, we give an alternative proof of Conjecture
\ref{trace-conj} and Theorem \ref{theo-trace} if $A$ is a Jacobian
$Jac(C)$. The proof is based on an explicit expression for
$P_{\varphi}(T)$ in terms of an $sncd$-model of the curve $C$.

Let $C$ be a smooth projective geometrically connected $K$-curve
of genus $g(C)$. An $sncd$-model of $C$ is a regular flat proper
$R$-model of $C$ whose special fiber is a divisor with strict
normal crossings. Let $\mathcal{C}$ be
 a relatively minimal $sncd$-model of $C$, with
$\mathcal{C}_s=\sum_{i\in I}N_iE_i$. We put
$\delta(C)=gcd\{N_i\,|\,i\in I\}$. For each $i\in I$ we put
$E_i^o=E_i\setminus \cup_{j\neq i}E_j$ and we denote by $d_i$ the
cardinality of $E_i\setminus E_i^o$. Let $A$ be the Jacobian
$Jac(C)$ of $C$. Then $A$ is tamely ramified iff $C$ is
cohomologically tame.

 We denote by
$\zeta_C(T)$ the reciprocal of the monodromy zeta function of $C$,
i.e.,
$$\zeta_C(T)=\prod_{i=0}^{2}\mathrm{det}(T\cdot Id- \varphi\,|\,H^i(X\times_K
K^t,\Q_\ell))^{(-1)^{i+1}}\ \in \Q_\ell(T)$$

\begin{theorem}\label{acampo}
Put $A=Jac(C)$. For each $i\in I$, we denote by $N_i'$ the
prime-to-$p$ part of $N_i$. We have
\begin{eqnarray}\label{eq-zeta1.1}
\zeta_C(T)&=&\prod_{i\in
I}(T^{N'_i}-1)^{-\chi_{\acute{e}t}(E_i^o)}
\\ P_{\varphi}(T)&=&(T-1)^2\prod_{i\in
I}(T^{N'_i}-1)^{-\chi_{\acute{e}t}(E_i^o)} \label{eq-zeta1.2}
\end{eqnarray}
 If $C$ is cohomologically tame, and either $\delta(C)$ is prime to $p$, or $g(C)\neq 1$, then
\begin{eqnarray}\label{eq-zeta2.1}
\zeta_C(T)&=&\prod_{i\in I}(T^{N_i}-1)^{-\chi_{\acute{e}t}(E_i^o)}
\\ P_{\varphi}(T)&=&(T-1)^2\prod_{i\in
I}(T^{N_i}-1)^{-\chi_{\acute{e}t}(E_i^o)} \label{eq-zeta2.2}
\end{eqnarray}
\end{theorem}
\begin{proof}
The expressions for $P_{\varphi}(T)$ follow immediately from the
expressions for $\zeta_C(T)$, since $\varphi$ acts trivially on
the degree $0$ and degree $2$ cohomology of $C$.

 By \cite[3.3]{abbes}, the complex of $\ell$-adic tame
nearby cycles $R\psi_\eta^t(\Q_\ell)$ associated to $\mathcal{C}$
can be endowed with a finite filtration such that the successive
quotients are tamely constructible in the sense of
\cite{Ni-trace}, and such that $\varphi$ has finite order on these
quotients. By \cite[6.1]{Ni-trace}, this implies that
\begin{equation}\label{eq-zetaint}
\zeta_C(T)=\int^{\times}_{\mathcal{C}_s}\zeta(\varphi\,|\,R\psi_\eta^t(\Q_\ell)_{*})
\end{equation}
where $\int^{\times}$ is multiplicative integration w.r.t. the
Euler characteristic, and
$$\zeta(\varphi\,|\,R\psi_\eta^t(\Q_\ell)_{*})$$ is the
constructible function on $\mathcal{C}_s$ mapping a closed point
$x$ of $\mathcal{C}_s$ to
\begin{equation}\label{eq-local}
\prod_{i=0}^{1}\mathrm{det}(T\cdot Id-
\varphi\,|\,R^i\psi_\eta^t(\Q_\ell)_x)^{(-1)^{i+1}} \in
\Q_\ell(T)\end{equation}
 Using the
local computation of the tame nearby cycles in \cite[I.3.3]{sga7a}
we see that the expression (\ref{eq-local}) equals
$(T^{N_i'}-1)^{-1}$ if $x\in E_i^o$ for some $i\in I$, and $1$
else. Computing (\ref{eq-zetaint}), we obtain the formula
(\ref{eq-zeta1.1}) for the zeta function (compare to A'Campo's
formula \cite{A'C} in the complex geometric setting).

Now assume that $C$ is cohomologically tame. If $g(C)\neq 1$, then
Saito's geometric tameness criterion \cite[3.11]{saito} implies
that $E_i^o\cong \mathbb{G}_{m,k}$ if $N_i\neq N'_i$, and hence
$\chi_{\acute{e}t}(E_i^o)=0$. Therefore,
 (\ref{eq-zeta2.1}) follows from (\ref{eq-zeta1.1}).
%$$\zeta_{C}(T)=\prod_{i\in I}(T^{N_i}-1)^{-\chi_{\acute{e}t}(E_i^o)}$$

Finally, assume that $C$ is cohomologically tame, $g(C)= 1$, and
$\delta(C)$ is prime to $p$.
%We denote by $\delta'(C)$ the prime-to-$p$ part of $\delta(C)$.
%Then $A=Jac(C)$ is a tamely ramified elliptic curve, and
%$$\zeta_{C}(T)=\zeta_A(T)$$
 By \cite[6.6]{liu-lorenzini-raynaud}, we know that $\delta(C)$ is
the order of the class of the torsor $C$ in $H^1(K,A)$, and that
the type of the model $\mathcal{C}$ is $\delta(C)$ times the type
of the minimal $sncd$-model of $A$. Combining
  Saito's criterion \cite[3.11]{saito} with the Kodaira-N\'eron
  reduction table for elliptic curves, we see that tameness of $C$
  implies that $\chi_{\acute{e}t}(E_i^o)=0$ for each $i\in I$ such that
  $N_i\neq N'_i$.
%
%By the first part of the proof, we find
%$$\zeta_{C}(T)=\zeta_A(T^{\delta'(C)})$$
%so that either $\delta'(C)=1$, or $\zeta_A(T)=1$, i.e. $A$ has
%semi-abelian reduction.
%
%  Hence, in order to prove that (\ref{eq-zeta2.2}) holds if $\delta(C)=\delta'(C)$,
%  we may assume that $\delta(C)=1$, so that $A=C$.
\end{proof}
\begin{remark}
Assume that $g(C)=1$, and denote by $\delta'(C)$ the prime-to-$p$
part of $\delta(C)$. If $A=Jac(C)$, then
$$\zeta_C(T)=\zeta_A(T)$$
On the other hand, (\ref{eq-zeta1.1}) and
\cite[6.6]{liu-lorenzini-raynaud} yield that
$$\zeta_C(T)=\zeta_A(T^{\delta'(C)})$$
So, either $\delta'(C)=1$, or $\zeta_A(T)=1$, i.e. $A$ has
semi-abelian reduction. We recover the well-known fact that
$H^1(K,E)$ is a $p$-group if $E$ is an elliptic curve with
additive reduction.
\end{remark}

As an immediate corollary, we obtain an alternative proof of
Theorem 2.1(i) in \cite{lorenzini-charpol}. Denote by $g(E_i)$ the
genus of $E_i$, for each $i\in I$. We put $a=\sum_{i\in I}g(E_i)$.
We denote by $\Gamma$ the dual graph of $\mathcal{C}_s$, and by
$t$ its first Betti number.
\begin{cor}[Lorenzini,
\cite{lorenzini-charpol}]\label{cor-lorenzini}  We have
$$P_{\varphi}(T)=(T-1)^{2a+2t}\prod_{i\in
I}\left(\frac{T^{N'_i}-1}{T-1}\right)^{2g(E_i)+d_i-2}$$
\end{cor}
\begin{proof}
This follows immediately from Theorem \ref{acampo}, the formula
$\chi_{top}(E_i^o)=2-2g(E_i)-d_i$, and the fact that
$2-2t=\sum_{i\in I}(2-d_i)$ (this is twice the Euler
characteristic of $\Gamma$).
\end{proof}
In \cite{lorenzini-charpol}, it was assumed that $\delta(C)=1$,
but our arguments show that this is not necessary. If
$\delta(C)=1$, then $a$ equals the abelian rank of $A=Jac(C)$ and
$t$ its toric rank (see \cite[p.\,148]{Lorenzini-jac}). We obtain
an alternative proof  of Conjecture \ref{trace-conj} and Theorem
\ref{theo-trace}.
% We recover the
%fact that the order of $1$ as a root of $P_{\varphi}(T)$ equals
%$2a+2t$ \cite[1.3]{Lenstra-Oort}.
\begin{cor}
Assume that $\delta(C)=1$, and put $A=Jac(C)$.
\begin{enumerate}
\item If $A$ does not have purely additive reduction, then
$P_{\varphi}(1)=0$. If, moreover, $A$ is tamely ramified, then
(Trace) holds for $A$.

\item If $A$ has purely additive reduction, then
$\phi'_A=P_{\varphi}(1)$. If, moreover, $A$ is tamely ramified,
then $\phi_A=\phi'_A$ and (Trace) holds for $A$.
\end{enumerate}
\end{cor}
\begin{proof}
%This is a special case of Proposition \ref{prop-notadd} and
%Theorem \ref{theo-trace}. Here we present a different proof, based
%on our expression for the monodromy zeta function.
%
(1) It follows from Corollary \ref{cor-lorenzini} that the order
of $1$ as a root of $P_{\varphi}(T)$ equals $2a+2t$. Hence, if $A$
does not have purely additive reduction, then $P_{\varphi}(1)=0$.
Combining this with Corollary \ref{varphi}, we see that (Trace)
holds for $A$ if $A$ is tamely ramified.

 (2)  By
\cite[1.5]{Lorenzini-jac}, we have $$\phi_A=\prod_{i\in
I}N_i^{d_i-2}$$ because the toric rank of $A$ is zero. By
Corollary \ref{cor-lorenzini} we know that
$$P_{\varphi}(T)=\prod_{i\in
I}\left(\frac{T^{N'_i}-1}{T-1}\right)^{d_i-2}$$ because $a=t=0$.
This yields
$$P_{\varphi}(1)=\prod_{i\in I}(N'_i)^{d_i-2}=\phi'_A$$
If $A$ is tamely ramified, then Saito's criterion
\cite[3.11]{saito} implies that $d_i=2$ if $N_i\neq N_i'$, so that
$$\phi'_A=\prod_{i\in I}(N'_i)^{d_i-2}=\prod_{i\in
I}N_i^{d_i-2}=\phi_A$$ Combining this with Corollary \ref{varphi},
we see that (Trace) holds for $A$.
%
%  Since $\varphi$ acts trivially on the degree $0$ and degree $2$
%cohomology of $X$, it follows from Lemma \ref{acampo} that
%$$P_{\varphi}(T)=\prod_{i\in I}(T^{N'_i}-1)^{-\chi_{\acute{e}t}(E_i^o)}(T-1)^2$$
%Note that the multiplicity of $1$ as a root of $P_{\varphi}(T)$ is
%given by
%$$2-\sum_{i\in I}\chi_{\acute{e}t}(E_i^o)=2-\sum_{i\in I}(2-2g(E_i)-d_i)=2-2\sum_{i\in I}g(E_i) - 2\chi_{top}(\Gamma)$$
%where $\Gamma$ denotes the dual graph of $\mathcal{C}_s$,
%$\chi_{top}(\Gamma)$ its topological Euler characteristic, and
%$g(E_i)$ the genus of $E_i$. The property that $A$ has purely
%additive reduction is equivalent to the property that $g(E_i)$
%vanishes for all $i\in I$ and the first Betti number
%$\beta_1(\Gamma)$ of $\Gamma$ is zero (see
%\cite[p.\,148]{Lorenzini-jac}). This implies
%$$2-\sum_{i\in I}\chi_{\acute{e}t}(E_i^o)=0$$  We recover the fact
%that $1$ is not a root of $P_{\varphi}(T)$ iff $A$ has purely
%additive reduction, and we obtain the expression
%$$P_{\varphi}(T)=\prod_{i\in I}\left(\frac{T^{N'_i}-1}{T-1}\right)^{d_i-2}$$
% Evaluating at $T=1$ yields
%$$P_{\varphi}(1)=\prod_{i\in I}(N'_i)^{d_i-2}$$
% which is equal to
%$\phi'_A$ since the toric rank of $A$ is zero
%\cite[1.5]{Lorenzini-jac}.
\end{proof}
%\begin{remark}
%Lemma \ref{acampo} yields a short alternative proof of
%\end{remark}
%\begin{remark}
%Note that, by Proposition \ref{acampo}, the multiplicity of $1$ as
%a root of $P_{\varphi}(T)$ is given by
%$$2-\sum_{i\in I}\chi_{\acute{e}t}(E_i^o)=2-\sum_{i\in I}(2-2g(E_i)-d_i)=2+2\sum_{i\in I}g(E_i) - 2\chi_{top}(\Gamma)$$
%where $\Gamma$ denotes the dual graph of $\mathcal{C}_s$ and
%$\chi_{top}(\Gamma)$ its topological Euler characteristic.
%\end{remark}
\section*{Acknowledgements}
\noindent The author is endebted to Q. Liu and D. Lorenzini for
valuable
 suggestions.

\end{document}